\documentclass[11pt]{amsart}
\usepackage{amssymb,latexsym}
\usepackage{amssymb,latexsym}
\usepackage{amsfonts,mathrsfs}
\usepackage[margin=1.4in]{geometry}
\usepackage{fancyhdr}
\usepackage{hyperref}
\usepackage{xcolor}

\newtheorem{theorem}{Theorem}[section]
\newtheorem{lemma}[theorem]{Lemma}
\newtheorem{proposition}[theorem]{Proposition}
\newtheorem{corollary}[theorem]{Corollary}
\newtheorem{definition}[theorem]{Definition}

\newtheorem*{conjecture}{Conjecture}
\theoremstyle{remark}
\newtheorem{remark}{Remark}[section]

\def\XXint#1#2#3{{\setbox0=\hbox{$#1{#2#3}{\int}$ }
\vcenter{\hbox{$#2#3$ }}\kern-.6\wd0}}

\newcommand{\Vol}{\mathrm{Vol}}
\newcommand{\Area}{\mathrm{Area}}

\makeatletter

\newcommand{\Rmnum}[1]{\expandafter\@slowromancap\romannumeral #1@}
\makeatother
\begin{document}
\allowdisplaybreaks

\title{ Alexandrov-Fenchel type inequalities for hypersurfaces in the sphere}
\author{Min Chen}
\address{1244 Burnside Hall, 805 Sherbrooke Street West Montreal, Quebec H3A 0B9}
\email{min.chen5@mail.mcgill.ca}
\thanks{\noindent \textbf{MR(2010)Subject Classification}   53C42, 35J60, 53C21}

\pagestyle{fancy}
\fancyhf{}
\renewcommand{\headrulewidth}{0pt}
\fancyhead[CE]{}
\fancyhead[CO]{\leftmark}
\fancyhead[LE,RO]{\thepage}

\begin{abstract}
The Alexandrov–Fenchel inequality, a far-reaching generalization of the classical isoperimetric inequality to arbitrary mixed volumes, is fundamental in convex geometry. In $\mathbb{R}^{n+1}$, it states: $\int_M\sigma_k d\mu_g \ge C(n,k)\big(\int_M\sigma_{k-1} d\mu_g\big)^{\frac{n-k}{n-k+1}}$. In \cite{Brendle-Guan-Li} (see also \cite{Guan-Li-2}), Brendle, Guan, and Li proposed a Conjecture on the corresponding inequalities in $\mathbb{S}^{n+1}$, which implies a sharp relation between two adjacent quermassintegrals: $\mathcal{A}_k(\Omega)\ge \xi_{k,k-1}\big(\mathcal{A}_{k-1}(\Omega)\big)$, for any $ 1\le k\le n-1$. This is a long-standing open problem.
In this paper, we prove a type of corresponding inequalities in $\mathbb{S}^{n+1}:$
$\int_{M}\sigma_kd\mu_g\ge \eta_k\big(\mathcal{A}_{k-1}(\Omega)\big)$ for any $0\le k\le n-1$. This is equivalent to the sharp relation among three adjacent quermassintegrals for hypersurfaces in $\mathbb{S}^{n+1}$(see (\ref{ineq three})),
which also implies a non-sharp relation between two adjacent quermassintegrals $\mathcal{A}_{k}(\Omega)\ge \eta_k\big(\mathcal{A}_{k-1}(\Omega)\big)$, for any $ 1\le k\le n-1$.

\end{abstract}

\maketitle
\numberwithin{equation}{section}
\section{Introduction}
Let $N^{n+1}(K)$ be the simply connected space form with constant sectional curvature $K=1$, $0$ or $-1$.
 Let $M$ be a closed hypersurface given by an embedding $X: M\rightarrow \mathbb{N}^{n+1}(K)$. Suppose $\Omega$ is the domain enclosed by $M$ in $\mathbb{N}^{n+1}(K).$
 
The $k$-th quermassintegral can be interpreted as the measure of the set of $k$-dimensional totally geodesic subspaces that intersect $\Omega$ in integral geometry. $\mathcal{A}_k$ is related to the curvature integral of the boundary using the Cauchy-Cronfton formulas.
\begin{definition}\label{def-integral}  Let $\kappa=\langle \kappa_1,\cdots, \kappa_n \rangle$ be the vector of $n$ principal curvatures of the hypersurface $M$. Denote the $k$-th elementary symmetric function of $\kappa$ by $\sigma_k(\kappa)$. 
The \textbf{$k$-th quermassintegal ${\mathcal A}_k$} is defined as follows:
\begin{align*}
&\mathcal{A}_{-1}(\Omega)=\Vol (\Omega),\\
&\mathcal {A}_0(\Omega)=\int_Md\mu_g=\Area(M),\\
&\mathcal {A}_1(\Omega)=\int_M\sigma_1d\mu_g+nK\Vol(\Omega),\\
&\mathcal {A}_k(\Omega)=\int_M\sigma_k d\mu_g+\frac{K(n-k+1)}{k-1}\mathcal{A}_{k-2}(\Omega),
\end{align*}
where $2\le k\le n.$
\end{definition}

In $\mathbb{R}^{n+1}$, the celebrated Alexandrov-Fenchel (\cite{Alexandrov1, Alexandrov2}) inequalities for convex hypersurfaces state that:
\begin{equation}\label{Eucli-A-F}
   \int_M\sigma_k d\mu_g \ge C(n,k)\big(\int_M\sigma_{k-1} d\mu_g\big)^{\frac{n-k}{n-k+1}},
\end{equation}
where $C(n,k)$ is a constant depending only on $n,k$ and is achieved if and only if M is a sphere. There has been some interest in extending the original Alexandrov–Fenchel inequality to
non-convex domains (see, \cite {Chang-Wang, Guan-Li-1, Trudinger}).

In $\mathbb{H}^{n+1}$, the Alexandrov-Fenchel inequalities have been extensively studied.  In the case of $h$-convexity, the full range of quermassintegral inequalities was obtained in \cite{Ge-Wang-Wu, Hu-Li-Wei, Wang-Xia}. In the case of positive sectional curvature, the relation between $\mathcal{A}_k$ and $\mathcal{A}_{-1}$ was proved in \cite{Andrews-Chen-Wei}. In the case of $k$-convexity and star-shapedness, the Minkowski-type inequality and the relation between $\mathcal{A}_2$ and $\mathcal{A}_0$ were obtained in \cite{Brendle-Hung-Wang, Li-Wei-Xiong} respectively; the relations between $\mathcal{A}_k$ and $\mathcal{A}_{l}$ for general $ -1\le k\le l\le n-1$ was established by Brendle-Guan-Li \cite{Brendle-Guan-Li} with an extra initial gradient
bound condition.

In $\mathbb{S}^{n+1}$, the Alexandrov-Fenchel inequalities for hypersurfaces in the sphere have been an open question for a long time.

Brendle, Guan, and Li \cite{Brendle-Guan-Li} (see also \cite{Guan-Li-2}) proposed a type of Alexandrov-Fenchel inequality in the following Conjecture.

\begin{conjecture}\label{Conj}
Let $M$ be a smooth, closed, connected, embedded, k-convex, star-shaped hypersurface in the sphere $\mathbb{S}^{n+1}$ enclosing a bounded domain $\Omega$. Then
\begin{equation}\label{Ak-Ak-1}
\mathcal{A}_k(\Omega)\ge \xi_{k,k-1}\big(\mathcal{A}_{k-1}(\Omega)\big),
\end{equation}
where $\xi_{k,k-1}$ is a unique positive function defined on $(0,s_{k-1})$ such that the equality holds when $M$ is a geodesic sphere. The equality holds if and only if $M$ is a geodesic sphere.
\end{conjecture}
Brendle, Guan, and Li introduced a locally constrained inverse curvature flow to study this Conjecture
\begin{equation}\label{B-G-L flow1}
X_t=\Big(\frac{\phi'(\rho)}{F}-\frac{u}{c_{n,k}}\Big)\nu,
\end{equation}
where $\phi, u, c_{n,k},\nu$ and $F$ are defined at the beginning of Section 2. However, only the case $k=n-1$ (when $M$ is convex) of (\ref{B-G-L flow1}) was confirmed in \cite{Brendle-Guan-Li} as the convergence of the flow (\ref{B-G-L flow1}) is challenging to prove for other classes. The main difficulty is the estimate of the lower bound of $F$. 

Chen, Guan, Li, and Scheuer \cite{Chen-Guan-Li-Scheuer} again introduced another fully nonlinear locally constrained curvature flow to study the inequalities (\ref{Ak-Ak-1})
\begin{equation}\label{B-G-L flow2}
X_t=\Big(c_{n,k}\phi'(\rho)-uF\Big)\nu.
\end{equation}
They proposed the same conjecture under the assumption of convexity.
\begin{conjecture}\label{Conj}
Let $M$ be a smooth, closed, connected, embedded, and convex hypersurface in the sphere $\mathbb{S}^{n+1}$ enclosing a bounded domain $\Omega$. Then
\begin{equation}\label{Ak-Ak-2}
\mathcal{A}_k(\Omega)\ge \xi_{k,k-1}\big(\mathcal{A}_{k-1}(\Omega)\big),
\end{equation}
where $\xi_{k,k-1}$ is a unique positive function defined on $(0,s_{k-1})$ such that the equality holds when $M$ is a geodesic sphere. The equality holds if and only if $M$ is a geodesic sphere.
\end{conjecture}

 However, the smooth convergence of the flow (\ref{B-G-L flow2}) is an open question since the $C^2$ estimate is still unknown.

Assuming that the smooth convergence of flow (\ref{B-G-L flow1}) and (\ref{B-G-L flow2}) has been established, the desired relation between two adjacent quermassintegrals $\mathcal{A}_k(\Omega)$ and $\mathcal{A}_{k-1}(\Omega)$ in the Conjecture can be proved. However, this is still an open question.

Recently, Makowski and Scheuer \cite{Makowski-Scheuer} proved the relation between $\mathcal{A}_{2k}$ and $\mathcal{A}_{0}$. The author and Sun \cite{Chen-Sun} then established the relation between $\mathcal{A}_k$ and $\mathcal{A}_{k-2}$. 
However, no substantial progress has been made on this problem since then.

As an analogous result of the Alexandrov-Fenchel inequalities (\ref{Eucli-A-F}) in $\mathbb{R}^{n+1}$, we prove a type of the corresponding Alexandrov-Fenchel inequalities in $\mathbb{S}^{n+1}$.
\begin{theorem}\label{sigma_k-A_k-1}
Let $M$ be an embedded, closed, connected and convex $C^2$-hypersurface in $\mathbb{S}^{n+1},$ then the following inequalities hold,

\begin{equation}\label{ineq1}
\int_{M}\sigma_kd\mu_g\ge \sqrt{\eta_k\big(\mathcal{A}_{k-1}(\Omega)\big)},\qquad \text{for any}  \quad 0\le k\le n-1 ,
\end{equation}
where $\eta_k$ is a positive function defined on $(0,s_{k-1})$ such that the equality holds when $M$ is a geodesic sphere. Here $s_{k-1}=\mathcal{A}_{k-1}(B_{\frac{\pi}{2}}(o))$. The equality holds if and only if $M$ is a geodesic sphere.

Moreover, when $k=1$, $\eta_1$ satisfies that
\begin{equation}\label{eta-1-ODE}
2\frac{n-1}{n}\eta_1(s)=(2ns+\eta_1'(s))s \quad \text{for} \quad s\in (0,s_0)
\end{equation}
and the explicit expression for $\eta_1$ is
\begin{equation}\label{eta-1}
\eta_1(s)=n^2(n+1)^{\frac{2}{n}}\omega_{n+1}^{\frac{2}{n}}s^{\frac{2(n-1)}{n}}-n^2s^2.
\end{equation}
when $k\ge 2$, $\eta_k$ satisfies that
\begin{equation}
\eta'_k(s)=\frac{\frac{2k(n-k)}{n-k+1}\eta_k(s)-2(n-k+1)\big(s-\frac{n-k+2}{k-2}\xi_{k-1,k-3}^{-1}(s)\big)^2}{k\big(s-\frac{n-k+2}{k-2}\xi_{k-1,k-3}^{-1}(s)\big)} \quad \text{for} \quad s\in (0,s_{k-1}),
\end{equation}
where $\xi_{k-1,k-3}$ is defined in \cite{Chen-Sun}.
\end{theorem}
\begin{remark}
When $k=1$ and $k=0$, (\ref{ineq1}) is the classical Minkowski inequality and isoperimetric inequality in $\mathbb{S}^{n+1}$ respectively. Thus, (\ref{ineq1}) can be viewed as all higher-order cases of the classical Minkowski and the isoperimetric inequality. 
\end{remark}

(\ref{ineq1}) is equivalent to the sharp relation among three adjacent quermassintegrals.
\begin{theorem}
  Let $M$ be an embedded, closed, connected and convex $C^2$-hypersurface in $\mathbb{S}^{n+1},$ then the following inequalities hold,
\begin{equation}\label{ineq three}
\mathcal{A}_{k}(\Omega)\ge\left\{\begin{split} & \eta_k\big(\mathcal{A}_{k-1}(\Omega)\big)+K\frac{n-k+1}{k-1}\mathcal{A}_{k-2}(\Omega),\qquad \text{for any}  \quad 2\le k\le n-1 ,\\
	& \eta_1\big(\mathcal{A}_{0}(\Omega)\big)+nK\Vol(\Omega),\qquad \text{for }  \quad k=1,
\end{split}
	\right.
\end{equation}
where $\eta_k$ is defined as above. The equality holds if and only if $M$ is a geodesic sphere.
\end{theorem}
It is natural to consider the relation among three adjacent quermassintegrals. For instance, the classical Alexandrov-Fenchel inequalities in convex geometry are the sharp relation among three adjacent quermassintegrals for hypersurfaces in $\mathbb{R}^{n+1}$
\begin{equation}\label{Eucli}
\mathcal{A}_k^2(\Omega)\ge \mathcal{A}_{k+1}(\Omega)\mathcal{A}_{k-1}(\Omega).
\end{equation}
Deriving inequalities that involve more than two quermassintegrals using geometric flow is often challenging. It is unknown whether the inequality (\ref{Eucli}) can be obtained for convex hypersurfaces in $\mathbb{R}^{n+1}$ by using the geometric flow method.

This also implies a non-sharp relation between two adjacent quermassintegrals.
\begin{corollary}
  Let $M$ be an embedded, closed, connected and convex $C^2$-hypersurface in $\mathbb{S}^{n+1},$ then the following inequalities hold,
\begin{equation}
\mathcal{A}_{k}(\Omega)\ge \eta_k\big(\mathcal{A}_{k-1}(\Omega)\big),\qquad \text{for any}  \quad 1\le k\le n-1 ,
\end{equation}
where $\eta_k$ is defined as above. 
\end{corollary}

Generally speaking, the main idea in applying curvature flow to prove geometric inequalities is to find monotone geometric quantities along the flow and then investigate the asymptotic of the limit. A nice feature of Brendle-Guan-Li' flow (\ref{B-G-L flow1}) is that the monotonicity of each quermassintegral can be obtained 
\begin{equation}\label{ineq2}
\partial_t \mathcal{A}_{l}(\Omega)\left\{\begin{split} & <0,\qquad \text{for } l>k ,\\
	& =0,\qquad \text{for }  l=k,\\
 &>0, \qquad \text{for } l<k ,
 \end{split}
	\right.
\end{equation}
along this flow. Similarly to Brendle-Guan-Li's flow (\ref{B-G-L flow1}), the monotonicity property for quermassintegrals holds as long as the flow (\ref{B-G-L flow2}) exists. When we try to use the locally constrained inverse curvature flow (\ref{B-G-L flow1}) and (\ref{B-G-L flow2}) to prove the inequalities in the Conjecture, the main difficulty is to prove the smooth convergence of the flow. 

Our approach is to apply the inverse curvature flow to study the Alexandrov-Fenchel-type inequalities
\[X_t=\frac{\nu}{F}.\]
The smooth convergence of the inverse curvature flow to the equator of the hemisphere has already been established by Gerhardt in \cite{Gerhardt} and Makowski-Scheuer in \cite{Makowski-Scheuer}.  For the inverse curvature flow, it is difficult to find a proper monotone quantity that can be used to prove the sharp relation between two adjacent quermassintegrals in the Conjecture. However, we use the inverse curvature flow to find a monotone quantity that can be used to obtain the sharp relation among three quermassintegrals, which implies a non-sharp relation between two adjacent quermassintegrals. The monotone quantity is fundamental when studying the geometric flow. The key in our proof is to find a novel monotone quantity $Q_k(t)$ along the inverse curvature flow. 
\[Q_k(t)=e^{-2\frac{k(n-k)}{n-k+1}t}\Big(\big(\int_{M(t)}\sigma_kd\mu_g\big)^2-\eta_k\big(\mathcal{A}_{k-1}(t)\big)\Big).\]
This allows us to derive the inequalities (\ref{ineq1}), which implies the optimal relation among three quermassintegrals.

The subsequent sections of this paper are organized as follows: in Section 2, we recall some general facts on the elementary symmetric functions and the quermassintegrals; in Section 3, we will introduce the application of inverse curvature flow to prove the inequalities. We will mainly focus on the monotonicity of the quantity $Q_k(t)$.

\section{Setting and general facts}
Let us present some basic facts which will be used later in this paper.

 Under the Gaussian geodesic normal coordinates with center in $o$, the metric can be expressed as
\[ds^2=d\rho^2+\phi^2(\rho)dz^2,\]
where $\phi(\rho)=\sin\rho,\rho \in [0,\pi)$ when $K=1$; $\phi(\rho)=\rho, \rho \in [0,\infty)$ when $K=0$; and $\phi(\rho)=\sinh\rho, \rho \in [0,\infty)$ when $K=-1,$ and $dz^2$ is the standard induced  metric on ${\mathbb S}^n$ in Euclidean space.
Let $M\subset \mathbb{N}^{n+1}(K)$ be a closed hypersurface and $\nu$ be the outward unit normal vector field. The function $u=\langle V,\nu \rangle$ is the generalized support function of the hypersurface. Assume $F=\frac{\sigma_{k+1}}{\sigma_{k}}(\kappa)$ and $c_{n,k}=\frac{\sigma_{k+1}}{\sigma_k}(I)$.

\begin{definition}(\cite{Guan})
For $1\le k\le n,$ let $\Gamma_k$ be a cone in $\mathbb{R}^n$ determined by
\[\Gamma_k=\{\lambda \in \mathbb{R}^n: \sigma_1>0,\cdots, \sigma_k>0\}.\]
An $n\times n$ symmetric matrix $W$ is called belonging to $\Gamma_k$ if $\lambda(W)\in \Gamma_k.$ Here we denote $\lambda(W)=(\lambda_1(W),\lambda_2(W),\cdots,\lambda_n(W))$ to be the eigenvalues of the symmetric matrix $W$.
\end{definition}
Then, we will introduce Newton-Maclaurin inequality.
\begin{lemma}\label{lemma-NM}(\cite{Guan})
For $W\in \Gamma_k,$
\[(n-k+1)(k+1)\sigma_{k-1}(W)\sigma_{k+1}(W)\le k(n-k)\sigma_k^2(W),\]
and
\[\sigma_{k+1}(W)\le c_{n,k}\sigma_k^{\frac{k+1}{k}}(W).\]
 The equality holds if and only if $W=cI$ for some $c>0.$
\end{lemma}

Next, we will present the evolution equations of $\sigma_l$ and quermassintegrals.
Let $M(t)$ be a smooth family of closed hypersurfaces in $\mathbb{N}^{n+1}(K).$ Let $X(\cdot,t)$ denote a point on $M(t).$  We will consider the flow
\begin{equation}\label{f}
X_t=f\nu.
\end{equation}
Along this flow, we have the following.

\begin{proposition} \label{evolu equa} (\cite{Brendle-Guan-Li},\cite{Reilly})
Under the flow (\ref{f}) for the hypersurface in a Riemannian manifold, suppose that $\Omega$ is the domain enclosed by the closed hypersurface; we have the following evolution equations.
\begin{align*}
\partial_t g_{ij}&=2fh_{ij},\\
\partial_t h_{ij}&=-\nabla_i\nabla_j f+f(h^2)_{ij}-fR_{\nu ij \nu},\\
\partial_t h^j_i&=-g^{jk}\nabla_k\nabla_j f-g^{jk}f(h^2)_{ki}-fg^{jk}R_{\nu ik \nu},\\
\partial_t \sigma_k& =\frac{\partial \sigma_k}{\partial h^j_i}\partial_t h^j_i,
\end{align*}
where $R_{\nu ik \nu}$ is the Riemannian curvature tensor in $\mathbb{N}^{n+1}(K)$.

Moreover, if $N$ has constant sectional curvature $K$, then for $l\ge 0,$ we have
\begin{equation}\label{evolution of sigma}
\partial_t \int_M\sigma_l=\int_M f[(l+1)\sigma_{l+1}-(n-l+1)K\sigma_{l-1}]d\mu_g,
\end{equation}
\end{proposition}

Using (\ref{evolution of sigma}) and the definition of $\mathcal{A}_k$, we have the following proposition.

\begin{proposition} (\cite{Brendle-Guan-Li},\cite{Guan-Li-2}) In $\mathbb{N}^{n+1}(K),$ along the flow (\ref{f}) for $0\le l<n-1,$ we have
\begin{equation}\label{evolu equa A_k}
\partial_t \mathcal{A}_l=(l+1)\int_Mf\sigma_{l+1}.
\end{equation}
\end{proposition}

Then, I will introduce the definition of $\eta_k$. We choose the origin $o$ as the center of this hemisphere.
Let $B_{\rho}(o)\subset \mathbb{N}^{n+1}(K)$ be the geodesic ball of radius $\rho$ centered at the origin $o$. Then we have
\[\frac{d}{d\rho}(\mathcal{A}_{k-1}(B_{\rho}(o)))=k\int_{\partial B_{\rho}(o)}\sigma_{k}>0,\]
for any $k=0,1,2,\cdots,n$. If we view $\mathcal{A}_k(B_{\rho}(o))$ as a function of $\rho$, then the inverse function can be denoted as
\[\rho=\xi_{k}(\mathcal{A}_{k-1}(B_{\rho}(o))),\]
where $\xi_k : (0,s_{k-1})\rightarrow (0,\pi/2)$ is a strictly increasing function for any fixed $k$. Here $s_{k-1}=\mathcal{A}_{k-1}(B_{\frac{\pi}{2}}(o))$. Then there exists a unique positive function $\eta_k$  defined on $(0,s_{k-1})$ such that
\begin{equation}\label{def-eta}
\int_{\partial B_{\rho}(o)}\sigma_kd\mu_g=\eta_{k}(\mathcal{A}_{k-1}(B_{\rho}(o))).
\end{equation}

 \vspace{.1in}
 
\section{the inverse curvature flow and applications}

In this section, we will use the inverse curvature flow to prove the main theorem.
Gerhardt in \cite{Gerhardt} and Makowski-Scheuer in \cite{Makowski-Scheuer} introduced the inverse curvature flows of strictly convex hypersurfaces in $\mathbb{S}^{n+1}$ and obtained smooth convergence of the flows to the equator.
\begin{theorem}(\cite{Gerhardt},\cite{Makowski-Scheuer})
Let $F\in C^{\infty}(\Gamma_n)$ be symmetric, monotone, concave, inverse concave, and homogeneous of the degree 1 curvature function. 
Assume $M(t)$ be the solutions of the flows
\begin{align}\label{e-ICF}
    X_t=\frac{\nu}{F},
\end{align}
where the initial hypersurfaces $M_0=M(0)$ are strictly convex. Then,
the flows exist on the maximal time interval $[0, T^*)$ and converge to the hemisphere's equator.
\end{theorem}

Next, we will study the properties of the function $\eta_k$, which will be used to prove the main theorem.
\begin{proposition}\label{prop eta}
For any $s\in (0,s_{k-1}),$ the following holds
\begin{equation}\label{deritive-eta}
\eta'_k(s)=\frac{\frac{2k(n-k)}{n-k+1}\eta_k(s)-2(n-k+1)\Big(s-\frac{n-k+2}{k-2}\xi_{k-1,k-3}^{-1}(s)\Big)^2}{k\Big(s-\frac{n-k+2}{k-2}\xi_{k-1,k-3}^{-1}(s)\Big)},
\end{equation}
where $\eta_k$  are defined as in (\ref{def-eta}).
\end{proposition}
\begin{proof} 

For $0< \rho(t)< \frac{\pi}{2}$, by the definition of $\eta_{k}$, we have
\begin{equation}\label{eta_k_t}
\big(\int_{\partial B_{\rho(t)}(o)}\sigma_kd\mu_g\big)^2= \eta_k\big(\mathcal{A}_{k-1}(B_{\rho(t)}(o))\big).
\end{equation}
Along the flow (\ref{e-ICF}) with $F=\frac{\sigma_k}{\sigma_{k-1}}$, using (\ref{evolution of sigma})  and (\ref{evolu equa A_k}), we have
\begin{align*}
0&=\frac{d}{dt}\Big(\big(\int_{\partial B_{\rho(t)}(o)}\sigma_kd\mu_g\big)^2- \eta_k\big(\mathcal{A}_{k-1}(B_{\rho(t)}(o))\big)\Big)\\
&=2\int_{\partial B_{\rho(t)}(o)}\sigma_kd\mu_g\int_{\partial B_{\rho(t)}(o)}\frac{\sigma_{k-1}}{\sigma_k}\Big((k+1)\sigma_{k+1}-(n-k+1)\sigma_{k-1}\Big)d\mu_g\\
&-\eta_k'\big(\mathcal{A}_{k-1}(B_{\rho(t)}(o))\big)\cdot k\int_{\partial B_{\rho(t)}(o)}\sigma_k\frac{\sigma_{k-1}}{\sigma_k}d\mu_g\\
&=2\frac{k(n-k)}{n-k+1}\big(\int_{\partial B_{\rho(t)}(o)}\sigma_kd\mu_g\big)^2-2(n-k+1)\int_{\partial B_{\rho(t)}(o)}\sigma_kd\mu_g\int_{\partial B_{\rho(t)}(o)}\frac{\sigma_{k-1}^2}{\sigma_k}d\mu_g\\
&-k\eta_k'\big(\mathcal{A}_{k-1}(B_{\rho(t)}(o))\big)\int_{\partial B_{\rho(t)}}\sigma_{k-1}d\mu_g\\
&=2\frac{k(n-k)}{n-k+1}\big(\int_{\partial B_{\rho(t)}(o)}\sigma_kd\mu_g\big)^2-2(n-k+1)\big(\int_{\partial B_{\rho(t)}(o)}\sigma_{k-1}d\mu_g\big)^2\\
&-k\eta_k'\big(\mathcal{A}_{k-1}(B_{\rho(t)}(o))\big)\int_{\partial B_{\rho(t)}}\sigma_{k-1}d\mu_g,\\
\end{align*}
where the last two steps follow from the fact that
\[(n-k+1)(k+1)\sigma_{k-1}(W)\sigma_{k+1}(W)\le k(n-k)\sigma_k^2(W)\]
and
\begin{equation}
\big(\int_{\partial B_{\rho(t)}(o)}\sigma_{k-1}d\mu_g\big)^2 \le \int_{M(t)}\sigma_kd\mu_g\int_{M(t)}\frac{\sigma^2_{k-1}}{\sigma_k}d\mu_g,
\end{equation}
with the inequality strict unless  $M(t)$ is a geodeic sphere. Hence, we see that the equality holds if and only if $M(t)$ is a geodesic sphere.

Case 1: When $k=1$, by Definition \ref{def-integral}, we know
\begin{align*}
&2\frac{n-1}{n}\big(\int_{\partial B_{\rho(t)}(o)}\sigma_1d\mu_g\big)^2-2n\mathcal{A}^2_{0}(B_{\rho(t)}(o))-\eta_1'(\mathcal{A}_{0}(B_{\rho(t)}(o))\mathcal{A}_{0}(B_{\rho(t)}(o)=0,
\end{align*}
By (\ref{eta_k_t}), we have
\[2\frac{n-1}{n}\eta_1(\mathcal{A}_0(B_{\rho(t)}(o)))-2n\mathcal{A}_0^2(B_{\rho(t)}(o))-\eta_1'(\mathcal{A}_0(B_{\rho(t)}(o)))\mathcal{A}_0(B_{\rho(t)}(o))=0.\]
We can obtain that
\[2\frac{n-1}{n}\eta_1(s)=(2ns+\eta_1'(s))s \quad \text{for} \quad s\in (0,s_0).\]

Case 2: When $k\ge 2$, by Definition (\ref{def-integral}), we know
\begin{align*}
&2\frac{k(n-k)}{n-k+1}\big(\int_{\partial B_{\rho(t)}(o)}\sigma_kd\mu_g\big)^2\\
&-2(n-k+1)\Big(\mathcal{A}_{k-1}(B_{\rho(t)}(o))-\frac{n-k+2}{k-2}\mathcal{A}_{k-3}(B_{\rho(t)}(o))\Big)^2\\
&-k\eta_k'(\mathcal{A}_{k-1}(B_{\rho(t)}(o)))\Big(\mathcal{A}_{k-1}(B_{\rho(t)}(o))-\frac{n-k+2}{k-2}\mathcal{A}_{k-3}(B_{\rho(t)}(o))\Big)=0.
\end{align*}
By the assumption (\ref{eta_k_t}),
we have
\begin{align*}
&2\frac{k(n-k)}{n-k+1}\eta_k\big(\mathcal{A}_{k-1}(B_{\rho(t)}(o))\big)\\
&-2(n-k+1)\Big(\mathcal{A}_{k-1}(B_{\rho(t)}(o))-\frac{n-k+2}{k-2}\mathcal{A}_{k-3}(B_{\rho(t)}(o))\Big)^2\\
&-k\eta_k'(\mathcal{A}_{k-1}(B_{\rho(t)}(o)))\Big(\mathcal{A}_{k-1}(B_{\rho(t)}(o))-\frac{n-k+2}{k-2}\mathcal{A}_{k-3}(B_{\rho(t)}(o))\Big)=0.
\end{align*}
To obtain the explicit formula of $\eta'_k$, all the geometric quantities above should be expressed as functions of $\mathcal{A}_{k-1}(B_{\rho(t)}(o))$. 
We need to use the fact in Theorem 1.2 in \cite{Chen-Sun}
\[\mathcal{A}_{k-1}(B_{\rho(t)}(o))=\xi_{k-1,k-3}\big(\mathcal{A}_{k-3}(B_{\rho(t)}(o))\big).\]
Then we have
\begin{align*}
&2\frac{k(n-k)}{n-k+1}\eta_k\big(\mathcal{A}_{k-1}(B_{\rho(t)}(o))\big)\\
&-2(n-k+1)\Big(\mathcal{A}_{k-1}(B_{\rho(t)}(o))-\frac{n-k+2}{k-2}\xi^{-1}_{k-1,k-3}\big(\mathcal{A}_{k-1}(B_{\rho(t)}(o))\big)\Big)^2\\
&-k\eta_k'(\mathcal{A}_{k-1}(B_{\rho(t)}(o)))\Big(\mathcal{A}_{k-1}(B_{\rho(t)}(o))-\frac{n-k+2}{k-2}\xi^{-1}_{k-1,k-3}\big(\mathcal{A}_{k-1}(B_{\rho(t)}(o))\big)\Big)=0.
\end{align*}

\begin{equation}
\eta'_k(s)=\frac{\frac{2k(n-k)}{n-k+1}\eta_k(s)-2(n-k+1)\Big(s-\frac{n-k+2}{k-2}\xi_{k-1,k-3}^{-1}(s)\Big)^2}{k\Big(s-\frac{n-k+2}{k-2}\xi_{k-1,k-3}^{-1}(s)\Big)} \quad \text{for} \quad s\in (0,s_{k-1}).
\end{equation}
Finally, we derive the explicit expression of $\eta_1$. We can solve the ODE (\ref{eta-1-ODE}) to obtain
\begin{equation}\label{eta-1}
\eta_1(s)=n^2(n+1)^{\frac{2}{n}}\omega_{n+1}^{\frac{2}{n}}s^{\frac{2(n-1)}{n}}-n^2s^2.
\end{equation}
\end{proof}

Now we can prove our main result.

\begin{proof}[Proof of Theorem \ref{sigma_k-A_k-1}]

We assume $M$ to be smooth and strictly convex. Otherwise, we can use convolutions as in the proof of Corollary 1.2 in \cite{Makowski-Scheuer} to obtain a sequence of approximating smooth strictly convex hypersurfaces converging in $C^2$ to $M$.

Case 1: When $k=0$,  inequality (\ref{ineq1}) follows by the classical isoperimetric inequality, which is proved in \cite{Guan-Li}. 

Case 2: When $k\ge 1$, let $M(t)$ solve the inverse curvature flow equation $X_t=\frac{\sigma_{k-1}}{\sigma_k}\nu$ with initial condition $M(0)=M.$ Assume $\Omega(t)$ is the domain enclosed by $M(t)$, denote $\mathcal{A}_k(t)=\mathcal{A}_k(\Omega(t))$.

Case 2.1: When $k=1$, we have
\begin{align*}
&\frac{d}{dt}\big((\int_{M(t)}\sigma_1d\mu_g)^2- \eta_1(\mathcal{A}_0(t))\big)\\
&=2\int_{M(t)}\sigma_1d\mu_g\int_{M(t)}\frac{1}{\sigma_1}(2\sigma_2-n\sigma_0)d\mu_g-\eta_1'(\mathcal{A}_0(t))\int_{M(t)}\sigma_1\frac{1}{\sigma_1}d\mu_g\\
&\le 2\frac{n-1}{n}(\int_{M(t)}\sigma_1d\mu_g)^2-2n\int_{M(t)}\sigma_1d\mu_g\int_{M(t)}\frac{1}{\sigma_1}d\mu_g-\eta'(\mathcal{A}_0(t))\mathcal{A}_0(t)\\
&\le 2\frac{n-1}{n}(\int_{M(t)}\sigma_1d\mu_g)^2-2n\mathcal{A}_0^2(t)-\eta_1'(\mathcal{A}_0(t))\mathcal{A}_0(t),\\
\end{align*}
where we have used the Newton-Maclaurin inequalities and Cauchy-Schwarz inequality-
ties in the last two steps. By (\ref{deritive-eta}), we have
\begin{align*}
&\frac{d}{dt}\big((\int_{M(t)}\sigma_1d\mu_g)^2- \eta_1(\mathcal{A}_0(t))\big)\\
&= 2\frac{n-1}{n}(\int_{M(t)}\sigma_1d\mu_g)^2-2n\mathcal{A}_0^2(t)-\frac{\frac{2(n-1)}{n}\eta_1\big(\mathcal{A}_0(t)\big)-2n\mathcal{A}^2_0(t)}{\mathcal{A}_0(t)}\mathcal{A}_0(t)\\
&= 2\frac{n-1}{n}\big((\int_{M(t)}\sigma_1d\mu_g)^2-\eta_1(\mathcal{A}_0(t))\big).\\
\end{align*}
We have
\[\frac{d}{dt}\Big( e^{-2\frac{n-1}{n}t}\big((\int_{M(t)}\sigma_1d\mu_g)^2- \eta_1(\mathcal{A}_0(t)\big)\Big)\le 0.\]
Denote
\[Q_1(t)=e^{-2\frac{n-1}{n}t}\big((\int_{M(t)}\sigma_1d\mu_g)^2-\eta_1(\mathcal{A}_0(t))\big),\]
then
\[\frac{d}{dt}Q_1(t)\le 0.\]
Thus
\[Q_1(t)-Q_1(0)\le 0,\]
for all $t\in [0,T^*).$\\
It is proved in Theorem 1.1 in (\cite{Gerhardt}) that the curvature flow converges to an equator, as $t\rightarrow T^*,$ and with
\[|\frac{\pi}{2}-\rho|_{m,\mathbb{S}^n}\le c_m\Theta \quad \forall t\in [t_{\delta},T^*),\]
where $\Theta=\arccos e^{t-T^*}$ and $|\frac{\pi}{2}-\rho|_{m,\mathbb{S}^n}=|\frac{\pi}{2}-\rho(\cdot,t)|_{C^m(\mathbb{S}^n)}$.
It follows that
\begin{align*}
&\Vol(\Omega_t)\rightarrow \Vol(B_{\frac{\pi}{2}}) ,\\
& \int_{M_t}d\mu_g\rightarrow |\Sigma(B_{\frac{\pi}{2}})|,\\
&\int_{M_t}\sigma_k d\mu_g\rightarrow 0 \quad \forall 1\le k\le n-1,
\end{align*}
as $t\rightarrow T^*.$ By the definition of $\mathcal{A}_k,$ we have
\[\mathcal{A}_k(t)\rightarrow \mathcal{A}_k(B_{\frac{\pi}{2}}(o))  \quad \forall  -1\le k\le n-1,\]
as $t\rightarrow T^*$, which implies that
\[\lim\limits_{t \to T^*}Q_1(t)=0.\]
Therefore, we have
\[(\int_{M(0)}\sigma_1d\mu_g)^2- \eta_1(\mathcal{A}_0(0))\ge 0,\]
i.e.,
\[(\int_{M}\sigma_1d\mu_g)^2- \eta_1(\mathcal{A}_0(\Omega))\ge 0,\]
the equality holds only if $M$ is a geodesic sphere. By the definition of $\eta_1$, we know that equality holds if $M$ is a geodesic sphere.

Case 2.2: When $k\ge 2$, we have
\begin{align*}
&\frac{d}{dt}\Big(\big(\int_{M(t)}\sigma_kd\mu_g\big)^2- \eta_k\big(\mathcal{A}_{k-1}(t)\big)\Big)\\
&=2\int_{M(t)}\sigma_kd\mu_g\int_{M(t)}\frac{\sigma_{k-1}}{\sigma_k}\Big((k+1)\sigma_{k+1}-(n-k+1)\sigma_{k-1}\Big)d\mu_g\\
&-k\eta_k'\big(\mathcal{A}_{k-1}(t)\big)\int_{M(t)}\sigma_k\frac{\sigma_{k-1}}{\sigma_k}d\mu_g\\
&\le 2\frac{k(n-k)}{n-k+1}\big(\int_{M(t)}\sigma_kd\mu_g\big)^2-2(n-k+1)\int_{M(t)}\sigma_kd\mu_g\int_{M(t)}\frac{\sigma^2_{k-1}}{\sigma_k}d\mu_g\\
&-k\eta_k'\big(\mathcal{A}_{k-1}(t)\big)\int_{M(t)}\sigma_{k-1}d\mu_g\\
&\le 2\frac{k(n-k)}{n-k+1}\big(\int_{M(t)}\sigma_kd\mu_g\big)^2-2(n-k+1)\big(\int_{M(t)}\sigma_{k-1}d\mu_g\big)^2\\
&-k\eta_k'\big(\mathcal{A}_{k-1}(t)\big)\int_{M(t)}\sigma_{k-1}d\mu_g,\\
\end{align*}
where we have used the Newton-Maclaurin inequalities and Cauchy-Schwarz inequalities in the last two steps. By (\ref{deritive-eta}), we have
\begin{align*}
&\frac{d}{dt}\Big(\big(\int_{M(t)}\sigma_kd\mu_g\big)^2- \eta_k\big(\mathcal{A}_{k-1}(t)\big)\Big)\\
& = 2\frac{k(n-k)}{n-k+1}\big(\int_{M(t)}\sigma_kd\mu_g\big)^2-2(n-k+1)\big(\int_{M(t)}\sigma_{k-1}d\mu_g\big)^2\\
&-k\frac{\frac{2k(n-k)}{n-k+1}\eta_k\big(\mathcal{A}_{k-1}(t)\big)-2(n-k+2)\Big(\mathcal{A}_{k-1}(t)-\frac{n-k+2}{k-2}\xi_{k-1,k-3}^{-1}\big(\mathcal{A}_{k-1}(t)\big)\Big)^2}{k\Big(\mathcal{A}_{k-1}(t)-\frac{n-k+2}{k-2}\xi_{k-1,k-3}^{-1}\big(\mathcal{A}_{k-1}(t)\big)\Big)}\\
&\cdot \int_{M(t)}\sigma_{k-1}d\mu_g\\
&= 2\frac{k(n-k)}{n-k+1}\big(\int_{M(t)}\sigma_kd\mu_g\big)^2-2(n-k+1)\big(\int_{M(t)}\sigma_{k-1}d\mu_g\big)^2\\
&-\frac{2k(n-k)}{n-k+1}\frac{\eta_k(\mathcal{A}_{k-1}(t))}{\Big(\mathcal{A}_{k-1}(\Omega(t))-\frac{n-k+2}{k-2}\xi_{k-1,k-3}^{-1}\big(\mathcal{A}_{k-1}(t)\big)\Big)}\cdot \int_{M(t)}\sigma_{k-1}d\mu_g\\
&+2(n-k+2)\Big(\mathcal{A}_{k-1}(t)-\frac{n-k+2}{k-2}\xi_{k-1,k-3}^{-1}\big(\mathcal{A}_{k-1}(t)\big)\Big)\cdot \int_{M(t)}\sigma_{k-1}d\mu_g.\\
\end{align*}
We claim that $\eta_k\big(\mathcal{A}_{k-1}(t)\big)> 0$ and $\mathcal{A}_{k-1}(t)-\frac{n-k+2}{k-2}\xi_{k-1,k-3}^{-1}\big(\mathcal{A}_{k-1}(t)\big)>0$ for any $t\in [0,T^*)$. In fact, 
\[\frac{d}{dt}\mathcal{A}_{k-1}(t)=k\int_{M(t)}\sigma_{k-1}d\mu_g>0.\]
Then

\[0< \mathcal{A}_{k-1}(0)\le \mathcal{A}_{k-1}(t)\le \lim\limits_{t \to T^*}\mathcal{A}_{k-1}(t)=\mathcal{A}_{k-1}(B_{\frac{\pi}{2}}(o))=s_{k-1}.\]
By the definition of $\eta_k$ and $\xi_{k-1,k-3}$, we have
\[\eta_k\big(\mathcal{A}_{k-1}(t)\big)>0,\] 
and 
\[\mathcal{A}_{k-1}(t)-\frac{n-k+2}{k-2}\xi_{k-1,k-3}^{-1}\big(\mathcal{A}_{k-1}(t)\big)>0\]
for any $t\in [0,T^*)$. Using the inequalities in Theorem 1.2 in \cite{Chen-Sun}
\[\mathcal{A}_{k-1}\big(\Omega\big)\ge\xi_{k-1,k-3}\big(\mathcal{A}_{k-3}\big(\Omega)\big),\]
we have
 \begin{align*}
&\frac{d}{dt}\Big(\big(\int_{M(t)}\sigma_kd\mu_g\big)^2- \eta_k\big(\mathcal{A}_{k-1}(t)\big)\Big)\\
&\le 2\frac{k(n-k)}{n-k+1}\big(\int_{M(t)}\sigma_kd\mu_g\big)^2-2(n-k+1)\big(\int_{M(t)}\sigma_{k-1}d\mu_g\big)^2\\
&-\frac{2k(n-k)}{n-k+1}\frac{\eta_k\big(\mathcal{A}_{k-1}(t)\big)}{\Big(\mathcal{A}_{k-1}(t)-\frac{n-k+2}{k-2}\mathcal{A}_{k-3}(t)\Big)}\cdot \int_{M(t)}\sigma_{k-1}d\mu_g\\
&+2(n-k+1)\Big(\mathcal{A}_{k-1}(t)-\frac{n-k+2}{k-2}\mathcal{A}_{k-3}(t)\Big)\cdot \int_{M(t)}\sigma_{k-1}d\mu_g\\
&= 2\frac{k(n-k)}{n-k+1}\big(\int_{M(t)}\sigma_kd\mu_g\big)^2-2(n-k+1)\big(\int_{M(t)}\sigma_{k-1}d\mu_g\big)^2\\
&-\frac{2k(n-k)}{n-k+1}\eta_k\big(\mathcal{A}_{k-1}(t)\big)+2(n-k+1) \big(\int_{M(t)}\sigma_{k-1}d\mu_g\big)^2\\
&=2\frac{k(n-k)}{n-k+1}\Big(\big(\int_{M(t)}\sigma_kd\mu_g\big)^2-\eta_k\big(\mathcal{A}_{k-1}(t)\big)\Big).\\
\end{align*}
Thus,
\[\frac{d}{dt}\Big[e^{-2\frac{k(n-k)}{n-k+1}t}\Big(\big(\int_{M(t)}\sigma_kd\mu_g\big)^2-\eta_k\big(\mathcal{A}_{k-1}(t)\big)\Big)\Big]\le 0.\]
Denote
\[Q_k(t)=e^{-2\frac{k(n-k)}{n-k+1}t}\Big(\big(\int_{M(t)}\sigma_kd\mu_g\big)^2-\eta_k\big(\mathcal{A}_{k-1}(t)\big)\Big),\]
then
\[\frac{d}{dt}Q_k(t)\le 0.\]
Thus
\[Q_k(t)-Q_k(0)\le 0,\]
for all $t\in [0,T^*).$
As $t\rightarrow T^*$, 
\[\lim\limits_{t \to T^*}Q_k(t)=0.\]
Therefore, we have
\[\big(\int_{M(0)}\sigma_kd\mu_g\big)^2-\eta_k\big(\mathcal{A}_{k-1}(0)\big)\ge 0,\]
i.e.,
\[\big(\int_{M}\sigma_kd\mu_g\big)^2- \eta_k\big(\mathcal{A}_{k-1}(\Omega)\big)\ge 0,\]
 the equality holds only if $M$ is the geodesic sphere. By the definition of $\eta_k$, we know that equality holds if $M$ is a geodesic sphere.

\vspace{.1in}

\end{proof}

\renewcommand{\abstractname}{Acknowledgements}
\begin{abstract}
This work was done when the author was a postdoc fellow at McGill University. The author would like to thank Prof. Pengfei Guan for his supervision and useful discussions.
\end{abstract}


\end{document}